\documentclass[leqno,a4paper,12pt]{article}
\usepackage{amsmath,amsthm,amssymb}
\usepackage[utf8]{inputenc}
\usepackage[T1]{fontenc}
\usepackage{enumerate}
\usepackage{bbm}
\usepackage{mathrsfs}
\usepackage{mathabx}
\usepackage{hyperref}
\usepackage{nicefrac}
\usepackage{tikz}
\usetikzlibrary{matrix}

\usepackage{url}
\makeatletter
\g@addto@macro{\UrlBreaks}{\UrlOrds}
\makeatother

\usepackage[sort,numbers]{natbib}

\providecommand{\noopsort}[1]{} 

\def\qed{\unskip\quad \hbox{\vrule\vbox
to 6pt {\hrule width 4pt\vfill\hrule}\vrule} }

\newcommand{\bez}{\nopagebreak\hspace*{\fill}
 \nolinebreak$\qed$\vspace{5mm}\par}

\newtheorem{Th}{Theorem}[section]
\newtheorem{Prop}[Th]{Proposition}
\newtheorem{Lemma}[Th]{Lemma}

\theoremstyle{definition}
\newtheorem{Remark}[Th]{Remark}
\newtheorem{Def}[Th]{Definition}
\newtheorem{Cor}[Th]{Corollary}

\newcommand{\beq}{\begin{equation}}
\newcommand{\eeq}{\end{equation}}

\def\scalar(#1,#2){(#1\mid#2)}

\newcommand{\raz}{\mathbbm{1}}

\newcommand{\cf}{{\cal F}}

\newcommand{\ov}{\overline}

\newcommand{\Z}{{\mathbb{Z}}}
\newcommand{\N}{{\mathbb{N}}}

\newcommand{\NN}{\mathbb{N}}
\newcommand{\ZZ}{\mathbb{Z}}

\newcommand{\Card}[1]{\mathopen\rvert #1 \mathclose\lvert} 
\DeclareMathOperator{\Supp}{supp} 
\newcommand{\DefAs}{:=}
\newcommand{\AsDef}{=:}

\newcommand{\Alphab}{\{ 0 , 1 \}} 
\newcommand{\Bset}{\mathscr{B}} 
\newcommand{\Bseq}{\eta} 
\newcommand{\BfreeSub}{X_{\eta}} 
\newcommand{\BadmSub}{X_{\Bset}} 
\newcommand{\Shift}{S}

\newcommand*{\BlCode}{\phi} 
\newcommand*{\Wind}{\ell} 

\newcommand{\Length}[1]{\lvert #1 \rvert} 
\newcommand{\Restr}[3]{#1(#2 , #3)} 

\title{On the Garden of Eden theorem \\for $\mathscr{B}$-free subshifts}
\author{G.~Keller, M.~Lema\'nczyk, C.~Richard, D.~Sell}
\begin{document}
\maketitle
\thispagestyle{empty}

\begin{abstract}
We prove that on $\mathscr{B}$-free subshifts, with $\mathscr{B}$ satisfying the Erd\"os condition, all cellular automata are determined by monotone sliding block codes. In particular,  this implies the validity of the Garden
 of Eden theorem for such systems.
\end{abstract}

\section{Introduction}

\subsection{Sets of multiples and \texorpdfstring{$\mathscr{B}$}{B}-free sets}

For a set $\mathscr{B}$ of natural numbers, consider its multiples $\mathcal M_{\mathscr B}=\bigcup_{b\in \mathscr B} b\mathbb Z$. Interest in sets of multiples arose in the 1930's from \textit{abundant integers}, i.e., integers $n$ whose sum of proper divisors is larger than or equal to $|n|$. It is readily seen that abundant integers are closed under multiplication by integers, hence they indeed constitute a set of multiples. See \cite[Sec.~11]{Dy-Ka-Ku-Le} for further properties. In the following we restrict, without loss of generality and without further mentioning, to \textit{primitive} $\mathscr B$, i.e., no element from $\mathscr B$ divides any other element from $\mathscr B$. We also exclude the trivial case $\mathscr{B}=\{1\}$  from now on.%
\footnote{This is done for convenience,  see the discussion in Section \ref{sec:Badm}.}
A rich class of sets of multiples is provided by the case of $\mathscr B$ being \textit{Erdös}, where $\mathscr B$ is assumed infinite which excludes periodic sets of multiples. Moreover, $\mathscr B$ is assumed coprime and \textit{thin}, i.e., $\sum_{b\in \mathscr{B}}1/b<\infty$. The latter guarantees that the multiples have a natural density and motivates the attribution to Erdös, see the discussion in Chapter 1 of \cite{Ha}.

\smallskip

Intimately related to sets of multiples are the sets $\mathcal F_{\mathscr B}=\mathbb Z\setminus \mathcal M_{\mathscr B}$ of \textit{$\mathscr B$-free integers}. In particular, the abundant integers yield the \textit{deficient integers}. Over the last years, $\mathscr B$-free integers have been intensely studied from a dynamical perspective, see \cite{Dy-Ka-Ku-Le} for a detailed account. Historically, the guiding example seems \textit{square-free integers}, which is obtained from the Erdös set  $\mathscr B=\{p^2: p \text{ prime}\}$. On one side, the influential article by  Sarnak \cite{Sa} listed, among other results, a number of dynamical properties of the subshift generated by the square-free integers. On the other side, square-free integers are pure point diffractive \cite{bmp00, Ri-St}, in spite of their apparent irregularity. As mathematical diffraction theory \cite{Babook} was developing from the 1990's onwards, a dynamical perspective on diffraction emerged, since pure point diffractivity had been shown to be equivalent to discrete dynamical spectrum of the subshift generated by the structure \cite{Ba-Le}, when equipped with its natural pattern frequency measure.
In fact, squarefree integers \cite[Ch.~10]{me73} and general $\mathscr{B}$-free integers are examples of Meyer's weak model sets, and this viewpoint was very fruitful not only for analysing the $\mathscr{B}$-free subshift dynamics \cite{Ka-Ke-Le, Ke, Ke-2020}, but also for developing a dynamical structure theory of general weak model sets \cite{Ke-Ri, Ke-Ri-St, KR19, Ke-2020}. Moreover,  $\mathscr{B}$-free subshifts admit natural multidimensional generalizations to lattices and number fields \cite{Ba-Hu, Ce-Vi, Dy1}, some of which may serve as useful models in physics.

\smallskip

Despite their exceptional naturality, $\mathscr{B}$-free subshifts display an impressive dynamical diversity varying from proximality in the Erdös case to minimality, which leads to the classical theory of Toeplitz subshifts \cite{Dy-Ka-Ku-Le}. In both cases their entropy can be both zero or positive. In this paper we deal with the proximal case, with initial focus on Erdös sets. In this latter case, on one hand, the associated dynamical systems exhibit properties characteristic of the full shift:  their entropy is positive, their set of invariant measures is the Poulsen simplex \cite{Ku-Le-We1, Ko-Ku-Kw}, and they are intrinsically ergodic \cite{Ku-Le-We}. Moreover, they are hereditary: if $x$ belongs to the subshift and if $y\leq x$ coordinatewise, then $y$ belongs to the subshift \cite{Ab-Le-Ru}. On the other hand, the unique measure of maximal entropy is not Gibbs \cite{Ku-Le(jr)}, a rather surprising property that was first discovered for the square-free subshift \cite{Pe}. Moreover, their automorphism group is trivial, see below.

\subsection{Endomorphisms of subshifts}

In recent years, some effort has been made in order to understand the endomorphism group ${\rm End}(X,S)$ and the automorphism group ${\rm Aut}(X,S)$ of a subshift $(X,S)$. In the proximal case, unless $|X|=1$, ${\rm End}(X,S)$ is larger than ${\rm Aut}(X,S)$, as the map sending each point to the fixed point of $S$ is a member of  ${\rm End}(X,S)$. Clearly, both ${\rm End}(X,S)$ and ${\rm Aut}(X,S)$ are valuable  topological conjugacy invariants. For Erdös $\mathscr{B}$-free subshifts $(X,S)$, Mentzen's theorem \cite{Me} states that ${\rm Aut}(X,S)$ is trivial, i.e., it consists only of the powers of the shift. Whereas Mentzen's proof was based on arithmetic arguments, an extension beyond the Erdös case by dynamical methods was recently given in \cite[Thm.~4.1]{Ke-2020}. See also \cite{Ba-Hu-Le} for a multidimensional generalisation. In the Toeplitz case, only partial results are available \cite{Dy, Dy-Ka-Ke}.

\smallskip

In this article we study endomorphisms of proximal $\mathscr{B}$-free subshifts. We adopt here language from the theory of cellular automata \cite{Ha-Mu} and call an endomorphism of a subshift $(X,S)$ a {\em Cellular Automaton} (CA). By the classical Curtis-Lyndon-Hedlund (CLH) theorem \cite[Thm.~6.2.9]{Li-Ma}, each CA $\tau:X\to X$ is given by a {\em (sliding block) code}  $\phi:\{0,1\}^{\{-\ell,\ldots,\ell\}}\to \{0,1\}$ for some $\ell\geq0$ via the formula\footnote{It is implicit here that we consider only those blocks $B\in\{0,1\}^{\{-\ell,\ldots,\ell\}}$ that appear in some $x\in X$.}
\beq\label{clh}
\tau(x)(n)=\phi(x(n-\ell,n+\ell))\text{ for all } x\in X\text{ and all }n\in\Z.
\eeq
It follows that ${\rm End}(X,S)$  is countable, which directed a lot of effort towards classifying all possible countable groups that can appear as invertible CA of a subshift \cite{Cy-Kr, Do-}. Let us restrict to hereditary subshifts. In that case, if a code $\phi$ is {\em monotone}, i.e., $\phi(B)\leq B(0)$ for each block $B\in\{0,1\}^{\{-\ell,\ldots,\ell\}}$, then the induced map $\tau$ is an endomorphism. (Note that monotone codes  for hereditary $\mathscr{B}$-free subshifts are already mentioned in the preprint version of \cite{Dy}.) The full shift shows that the converse does not hold: it is hereditary, but there are CA, considered modulo the powers of the shift, which do not arise from monotone codes. The following two questions are natural:
\begin{itemize}
\item[(i)] When is every CA given by a monotone code, up to a power of the shift?
\item[(ii)] How to read off non-invertibility properties of a CA from its code?
\end{itemize}

\subsection{Results of the paper}

Let $\tau$ be a CA on $(X,S)$. The elements of $X\setminus \tau(X)$ are called {\em Garden of Eden (GoE) configurations} in the theory of cellular automata, see \cite{Ce-Co} for a recent review. The name arises from the fact that such configurations can only occur at time $0$, when considering iterates $\tau^{k}(X_0)$ of an initial set $X_0\subset X$. A surjective CA does not admit any GoE configuration. In that context, a crucial property of a CA is a weakened form of injectivity: A CA is {\em pre-injective} if it is injective on every collection of configurations that mutually differ in at most finitely many coordinates. One says \textit{``the Garden of Eden theorem holds on some subshift''}, if for any CA on this subshift surjectivity is equivalent to pre-injectivity. The implication that surjectivity implies pre-injectivity is also called the {\em Moore property}, the reverse implication is called the {\em Myhill property}.
It is known that the GoE holds for irreducible subshifts of finite type \cite[Thm.~7.15]{Ce-Co} and that any strongly irreducible subshift satisfies the Myhill property \cite[Thm.~7.10]{Ce-Co}.
In our setting, which is not subsumed by the above results, the GoE theorem holds in particular for the class of Erd\"os $\mathscr{B}$-free subshifts.  In fact, our main result in the Erdös case even answers the two questions raised above.

\begin{Th}\label{t:main} Assume that $\mathscr{B}$ is an Erd\"os set. Then
the corresponding $\mathscr{B}$-free subshift satisfies the GoE theorem. If a CA is onto, then it equals a power of the shift. Moreover, every CA is given by a monotone code, modulo a power of the shift.
\end{Th}

Our result states a very strong form of the Moore property, which implies Mentzen's theorem that the automorphism group of an Erdös $\mathscr{B}$-free subshift is trivial.

\smallskip

Subsequently, the paper discusses extensions beyond the Erdös case. One extension is motivated by the well-known observation that, in the Erdös case, the $\mathscr{B}$-free subshift $(X_\eta,S)$ coincides with the so-called $\mathscr{B}$-admissible subshift $(X_\mathscr{B},S)$, which is arithmetic in nature and explained in Section~\ref{s:initial}. In that sense the Erdös case is ``very arithmetic'' from a dynamical point of view. Theorem~\ref{t:GoEforBehrend} below states the GoE theorem and the extension of Mentzen's theorem for the setting $X_\eta=X_\mathscr{B}$, which reduces to Theorem~\ref{t:main} in the Erdös setting.

\smallskip

In general (remembering that $\mathscr{B}\neq\{1\}$), the condition $X_\eta=X_\mathscr{B}$ implies that $\mathscr{B}$ is infinite and coprime  \cite{Ka-Le}.  In that setting, a positive topological entropy of the $\mathscr{B}$-free subshift is equivalent to $\mathscr{B}$ being thin \cite{Ab-Le-Ru}. Hence positive topological entropy characterises the Erd\"os case in that setting. The case of zero topological entropy corresponds to the coprime \textit{Behrend} case, i.e., to the $\mathscr{B}$-free sets of density zero where $\mathscr{B}$ is coprime. Examples of the latter type are discussed in \cite{Ka-Le}.

\smallskip

The above setting is subsumed by the class of hereditary $\mathscr{B}$-free subshifts. Using arithmetic arguments, we extend Mentzen's theorem to general hereditary $\mathscr{B}$-free subshifts in Theorem~\ref{t1}, which contains \cite[Thm.~4.1]{Ke-2020} as a special case. Concerning GoE properties, we can prove the Myhill property if the subshift has a unique measure of maximal entropy of full topological support, see Proposition~\ref{t:goe1}. (In fact the Myhill property holds for an arbitrary hereditary subshift of the latter form.) We can prove the Moore property under a certain non-degeneracy assumption on the CA, see Lemma~\ref{l:NotAllZeroMoore}.

Let us restrict to hereditary $\mathscr{B}$-free subshifts where $\mathscr{B}$ is taut, see Section~\ref{sec:taut} below for a definition. By \cite[Thm.~3]{Ke}, heredity is equivalent to proximality in the taut case. As thin implies taut, this also generalises the Erdös case. Here, the situation is more transparent, as the first two assertions of Theorem~\ref{t:main} hold and we can characterise the surjectivity in several ways, see Theorem~\ref{Thm:tautprox}. An example (which is not Erd\"os) in that class is given by the deficient integers, compare \cite[Sec.~11]{Dy-Ka-Ku-Le}.

\smallskip

We finally discuss the general proximal case, i.e., $\mathscr{B}$ contains an infinite coprime set \cite[Thm.~B]{Dy-Ka-Ku-Le}. This case contains all Behrend $\mathscr{B}$-free subshifts, which are characterised by being proximal and having zero entropy.
In fact, we can reduce the analysis of the automorphism group to that of a naturally associated \textit{hereditary} subsystem, which is a taut $\mathscr{B}'$-subshift, without changing the density. Recalling that $\mathscr{B}$ is Behrend if the density of $\mathscr{B}$-free integers is zero, it seems appropriate to call this reduction the \textit{relative Behrend} case.

\subsection{Plan of the paper}

The paper is organized follows. In a preparatory Section~\ref{s:initial} we recall some basic theory, concentrating on various equivalences between arithmetic properties of $\mathscr{B}$ and topological properties of the corresponding $\mathscr{B}$-free subshift. Section~\ref{sec:moore} discusses the Moore property for $\mathscr{B}$-admissible subshifts. Section~\ref{sec:myhill} discusses the Myhill property for classes of hereditary subshifts. This leads to a proof of Theorem~\ref{t:GoEforBehrend}, which was discussed in the previous subsection. In Section~\ref{s:mentzen} we show that the heredity of a $\mathscr{B}$-free subshift forces any of its automorphisms to be trivial. This result applies in particular to the Erdös case, and more generally to the proximal case, where $\mathscr{B}$ is additionally assumed to be taut. Section~\ref{sec:hertaut} is devoted to a systematic study of the proximal taut case, which is described by Theorem~\ref{Thm:tautprox}.
In Section~\ref{s:direction} we briefly discuss  CA on general proximal $\mathscr{B}$-free subshifts and formulate some open problems.

\section{\texorpdfstring{$\mathscr{B}$}{B}-free dynamical systems}\label{s:initial}

Here we collect properties of $\mathscr{B}$-free and $\mathscr{B}$-admissible subshifts for later use, which are mainly taken from \cite{Dy-Ka-Ku-Le}. For more general background we refer to the monographs \cite{Ha} and \cite{Li-Ma, Wa, Down}.

\subsection{Some generalities on subshifts}

Consider the compact space $\{0,1\}^{\Z}$ of bi-infinite $0-1$-sequences, equipped with the usual product metric $d$. The {\em support} of $x\in \{0,1\}^{\Z}$ is the set ${supp}(x)=\{j\in\Z:\: x(j)\ne0\}$.
The {\em (left) shift}  $S:\{0,1\}^{\Z}\to \{0,1\}^{\Z}$, defined by $(Sx)(j)=x(j+1)$ for all $j\in\Z$, is a homeomorphism. If $X\subset\{0,1\}^{\Z}$, which we always assume to be non-empty, is closed and $S$-invariant then the resulting dynamical system $(X,S)$ is called a {\em subshift}. $X$ is called {\em hereditary},  if $y\in \{0,1\}^{\Z}$ belongs to $X$ whenever there is some $x\in X$ such that $y\leq x$ coordinatewise. Denote by $\widetilde{X}\subset \{0,1\}^{\Z}$ the smallest hereditary set containing $X$. Given a subshift $(X,S)$, we call the subshift $(\widetilde{X},S)$ its {\em hereditary closure}. We call $x,y\in X$ {\em asymptotic} if $d(S^rx,S^ry)\to 0$ for $|r|\to\infty$. It is easily seen that $x,y$ are asymptotic if and only if they differ in at most finitely many coordinates. If $d(S^{r_k}x,S^{r_k}y)\to 0$ along a sequence $(r_k)$ that tends to infinity, then $x,y$ are called {\em proximal}. A subshift $(X,S)$ is called {\em proximal} if any two elements from $X$ are proximal.


Denote by $M(X,S)$ the simplex of $S$-invariant probability measures on $(X,S)$, which is non-empty by the Krylov-Bogolioubov theorem, see e.g.~\cite[Cor.~6.9.1]{Wa}. Each $\nu\in M(X,S)$ is determined by its values $\nu(B)$ on all finite {\em blocks} or {\em words} $B\in \bigcup_{m=1}^\infty \{0,1\}^{m}$. More precisely, we identify any block $B\in \{0,1\}^m$, where $|B|=m$ is the {\em length} of $B$, with the corresponding cylinder set $\{x\in X:\: x(0,m-1)=B\}$, and thus the set of all finite blocks generates the Borel $\sigma$-algebra on $X$.
Given $\nu\in M(X,S)$ and a sequence $(N_k)$, the element $x\in X$ is {\em generic along $(N_k)$ for $\nu$} if
\beq\label{gener1}
\int f\,d\nu=\lim_{k\to\infty}\frac1{N_k}\sum_{n=1}^{N_k}f(S^nx)
\eeq
for each $f\in C(X)$. If $(N_k)=(k)$, then $x$ is called {\em generic for $\nu$}. Recall that by the variational principle, see e.g.~\cite[Sec.~7.4]{Down} or \cite[Thm.~8.6]{Wa}, the topological entropy $h(X,S)$ is the supremum of the measure-theoretic entropies $h_\nu(X,S)$ with $\nu\in M(X,S)$. Note that for a subshift the supremum is in fact attained, see e.g.~\cite[Fact~7.2.4]{Down}  or \cite[Thms.~8.2, 8.7]{Wa}.  A subshift $(X,S)$ is called {\em intrinsically ergodic} \cite{We0} if there is precisely one measure of maximal entropy.

\
\subsection{\texorpdfstring{$\mathscr{B}$}{B}-free subshifts and their hereditary closures}

We recall  that throughout we assume that $\mathscr{B}\subset \N$ does not contain 1 (and is primitive), if not stated otherwise. Consider a set $\mathcal{F}_{\mathscr{B}}=\Z\setminus \mathcal M_\mathscr{B}$ of $\mathscr B$-free integers. By the Davenport-Erdös theorem, see e.g.~\cite[Thm.~0.2]{Ha}, $\mathscr{B}$-free integers possess the logarithmic density
\begin{displaymath}
\delta(\cf_{\mathscr{B}})=\lim_{N\to\infty}\frac1{\log N}\sum_{n\in\cf_{\mathscr{B}}\cap\{1,\ldots, N\}} \frac{1}{n}\ ,
\end{displaymath}
which coincides with its natural upper density. In the Erdös case, the natural density exists, see e.g.~\cite[Sec.~1.1]{Ha}. Let $\eta=\raz_{\cf_{\mathscr{B}}}\in \{0,1\}^\Z$ denote the characteristic function of $\mathcal{F}_{\mathscr{B}}$ and
define $X_\eta=\ov{\{S^r\eta:\: r\in\Z\}}\subset\{0,1\}^{\Z}$.
The system $(X_\eta,S)$ is the corresponding $\mathscr{B}$-{\em free subshift}.   Due to its particular arithmetic origin, the point $\eta$ enjoys the following {\em maximality} property\footnote{Maximality arguments have been used before e.g.~in \cite[p.~186]{Ba-Hu}.}. As this property is crucial in later arguments, we provide its short proof.

\begin{Lemma}\label{maksymalnosc}
Let $(X_\eta, S)$ be a $\mathscr B$-free subshift. If $\eta\leq y\in X_\eta$ then $\eta=y$.
\end{Lemma}

\begin{proof}
Assume for a contradiction that $\eta\neq y$. Then ${\rm supp}(y)$ contains some integer $m$ that is not $\mathscr{B}$-free. Hence  $m=bn$ for some $b\in\mathscr{B}$ and some $n\in\Z$.  Since $\eta\leq y$, we have in particular $|{\rm supp}(\eta) \text{ mod }b|< |{\rm supp}(y) \text{ mod }b|$. This is contradictory, as there is a finite word $w$ in $y$ such that $ |{\rm supp}( w ) \text{ mod } b | = | {\rm supp}(y) \text{ mod } b |$. Indeed, for every $ r \in {\rm supp}(y) \text{ mod } b$ choose first a position $ j \in r+b\Z$ such that $ y(j)=1$. Then take a finite word in $y$ containing all $j$. As $ y \in X_\eta $, the word $w$ is also contained in $\eta$. Hence $|{\rm  supp}( \eta ) \text{ mod }b | \geq |{\rm supp}( w ) \text{ mod }b|$.
\end{proof}

\smallskip

A distinguished measure on $(X_\eta, S)$ is the so-called {\em Mirsky measure} $\nu_\eta$.  The point $\eta$ is generic along a subsequence for $\nu_\eta$, see \cite[Prop.~E]{Dy-Ka-Ku-Le}. In fact, $\eta$ is generic for $\nu_\eta$ in the logarithmic sense, i.e., if we replace  in \eqref{gener1}  the Ces\`aro summation by the logarithmic one \cite[Thm.~4.1, Thm.~2.23]{Dy-Ka-Ku-Le}. In the Erd\"os case, $\eta$ is generic in the usual sense. The Mirsky measure is the push-forward of the Haar measure on an odometer canonically associated to $\mathscr B$. In particular, the system $(X_\eta,\nu_\eta,S)$ is a zero entropy system, which is ergodic and has discrete spectrum \cite[Thm.~F]{Dy-Ka-Ku-Le}.

\smallskip

For a general $\mathscr{B}$-free subshift, no formula for its topological entropy is known. However, a non-trivial upper bound can be given using the description as a weak model set \cite[Thm.~4.5]{Hu-Ri}. The situation changes if we pass to the  hereditary closure $\widetilde{X}_\eta$ of a $\mathscr{B}$-free subshift. This can then be used to obtain a formula for the case of a proximal $\mathscr{B}$-free subshift, as shown in Corollary~\ref{c:entfor} below. In that case the upper bound from \cite{Hu-Ri} is attained. The following results appear as Proposition~K, Theorem J and Remark 10.2 in   \cite{Dy-Ka-Ku-Le}. The proofs rely on earlier results for the Erdös case \cite{Ku-Le-We} and for squarefree integers \cite{Pe}.

\begin{Th}[\cite{Ku-Le-We},\cite{Dy-Ka-Ku-Le}] \label{atw4}
Consider a $\mathscr{B}$-free subshift $(X_\eta,S)$. Then:
\begin{displaymath}
h(\widetilde{X}_\eta,S)=d\cdot \log2,
\end{displaymath}
where $d=\nu_\eta(1)$ is the logarithmic density of 1's in $\eta$, and
\begin{displaymath}
\mbox{$(\widetilde{X}_\eta,S)$ is intrinsically ergodic}
\end{displaymath}
with the measure of maximal entropy $\nu_\eta\ast B(\frac12,\frac12)$, where $B(\frac12,\frac12)$ stands for the $(\frac12,\frac12)$-Bernoulli measure on $\{0,1\}^{\Z}$ and the multiplicative convolution above is the image of the product measure $\nu_\eta\otimes B(\frac12,\frac12)$ via the coordinatewise multiplication map.\footnote{Note that the image of this map is contained in $\widetilde{X}_\eta$.} \qed
\end{Th}

\subsection{Tautness in \texorpdfstring{$\mathscr B$}{B}-free subshifts}\label{sec:taut}

The notion of tautness, which arises naturally in the theory of sets of multiples \cite[Def.~0.11]{Ha}, also turned out to be important for $\mathscr{B}$-free sets. Recall that $\mathscr{B}$ is {\em taut} if
$$
\delta(\cf_{\mathscr{B}})< \delta(\cf_{\mathscr{B}\setminus\{b\}})\text{ for each }b\in\mathscr{B}.
$$
It is easy to see that $\mathscr{B}$ thin implies $\mathscr B$ taut. In particular, the Erd\"os case is taut. From a dynamical point of view, the role of tautness is seen in the following three results. The first one combines  \cite[Cor.~1.11]{Ku-Le(jr)} with \cite[Thm.~2]{Ke}, which is based on a measure-theoretic characterisation of tautness in terms of weak model sets \cite[Thm.~A]{Ka-Ke-Le}.

\begin{Th}[\cite{Ke},\cite{Ku-Le(jr)}]
\label{atw1}
$\mathscr{B}$ is taut if and only if ${\rm supp}(\nu_\eta)=X_\eta$. \qed

\end{Th}

The second result is Theorem C in \cite{Dy-Ka-Ku-Le}.

\begin{Th}[\cite{Dy-Ka-Ku-Le}]\label{atw3}
For any $\mathscr{B}$ there exists a unique taut $\mathscr{B}'$
such that $\eta'\leq \eta$, $X_{\eta'}\subset \widetilde{X}_\eta$, $\nu_\eta=\nu_{\eta'}$, $M(\widetilde{X}_\eta,S)=M(\widetilde{X}_{\eta'},S)$, in particular,
$h(\widetilde{X}_\eta,S)=h(\widetilde{X}_{\eta'},S)$. \qed
\end{Th}

The third result combines \cite[Thm.~B]{Dy-Ka-Ku-Le} and \cite[Thm.~3]{Ke}. It relates proximality, which is an invariant of topological conjugacy, to heredity, which serves as a useful tool for analysis of $\mathscr{B}$-free sets, but lacks invariance in general.

\begin{Th}[\cite{Ke}]\label{atw2} Assume that $\mathscr{B}$ is taut. Then
$(X_\eta,S)$ is proximal if and only if $(X_\eta,S)$ is hereditary. \qed
\end{Th}

\subsection{Proximal \texorpdfstring{$\mathscr{B}$}{B}-free subshifts and Behrend subshifts}\label{sec:probeh}
Several characterizations of proximality of $\mathscr{B}$-free subshifts can be found in \cite[Thms.~3.7, 3.12]{Dy-Ka-Ku-Le} and \cite[Thm.~C]{Ka-Ke-Le}. We mention here that proximality is equivalent  to the property that the all zero sequence $0^{\Z}$ belongs to $X_\eta$, or to the existence of an infinite coprime  subset of $\mathscr{B}$. Let us call $\mathscr B$ {\em proximal} if the induced $\mathscr B$-free subshift is proximal.
A particular class of proximal systems is obtained when we consider Behrend sets \cite[Def.~0.6]{Ha}, i.e., sets that satisfy  $\delta(\cf_{\mathscr{B}})=0$. It is easily seen that in the Behrend case the Mirsky measure $\nu_{\eta}$ is just the Dirac measure $\delta_{0^{\Z}}$. In fact, in this case, $(X_\eta,S)$ is proximal and uniquely ergodic \cite[Rem.~6.3]{Dy-Ka-Ku-Le}. In particular, it has zero entropy.

\smallskip

Let us recall the following result, which is implicit in the proof of Eqn.\ (25) on page 13 in \cite{Ku-Le(jr)}. Its proof is a rather direct consequence of  Theorem~\ref{atw1}.

\begin{Prop}[\cite{Ku-Le(jr)}]\label{stw1}
For any $\mathscr{B}$ and the unique taut $\mathscr{B}'$ as in Theorem~\ref{atw3}, we have  $X_{\eta'}\subset X_\eta$. \qed
\end{Prop}

This can be used to obtain the following new result about proximal $\mathscr B$-free systems. The proximal case will further be studied in Section~\ref{s:direction}.

\begin{Cor} \label{c:entfor}
If $(X_\eta,S)$ is proximal, then
$$
h(X_\eta,S)=d\cdot \log 2,
$$
where $d=\nu_\eta(1)$. In particular, $h(X_\eta,S)=0$ if and only if $\mathscr{B}$ is Behrend.   Moreover, each proximal $\mathscr{B}$-free subshift is intrinsically ergodic,  and the measure of maximal entropy is $\nu_\eta\ast B(\frac12,\frac12)$.
\end{Cor}

\begin{proof} Let $\mathscr{B}'$ be the unique taut as in Theorem~\ref{atw3}. Then
\begin{displaymath}
h(\widetilde{X}_{\eta'},S)=h(\widetilde{X}_\eta,S).
\end{displaymath}
Now, $X_{\eta'}\subset X_\eta$ in view of Proposition~\ref{stw1}, so $(X_{\eta'},S)$ is, as a subsystem of a proximal system, also proximal. But $\mathscr{B}'$ is taut, so by Theorem~\ref{atw2}, $(X_{\eta'},S)$ is hereditary. Now, by Theorem~\ref{atw4} (used twice) and remembering that $\nu_\eta(1)=\nu_{\eta'}(1)$, we obtain
$$
d\cdot \log 2=h(X_{\eta'},S)\leq h(X_\eta,S)\leq h(\widetilde{X}_\eta,S)=d\log2$$
and the first claim follows. The intrinsic ergodicity in the proximal case follows from  Theorems~\ref{atw3},~\ref{atw2},~\ref{atw4}  and Proposition~\ref{stw1}.
\end{proof}

\subsection{\texorpdfstring{$\mathscr{B}$}{B}-admissible subshifts}\label{sec:Badm}

Recall that $A\subset\Z$ is $\mathscr{B}$-{\em admissible} if, for every $b\in \mathscr{B}$, its set of residue classes $A \mod b=\{a \mod b: a\in A\}$ contains less than $b$ elements. Analogously, a word is called $\mathscr{B}$-admissible if its support is $\mathscr{B}$-admissible. Then a subset of \( \{ 0,1 \}^{\ZZ} \) is $\mathscr{B}$-admissible if and only if all its words are $\mathscr{B}$-admissible. Note that $\mathscr{B}$-admissibility is stable under shifting. Clearly, any $\mathscr{B}$-free set is $\mathscr{B}$-admissible, as $0$ is a missing residue class for each $b\in\mathscr{B}$.
The shift invariant and closed space
$$
X_{\mathscr{B}}:=\{x\in\{0,1\}^{\Z}:\:{\rm supp}(x)\text{ is }\mathscr{B}\text{-admissible}\}
$$
was introduced in \cite{Sa}. The subshift $(X_{\mathscr{B}},S)$ is called the $\mathscr{B}$-{\em admissible} subshift \cite{Dy-Ka-Ku-Le}, see also the discussion of admissibility in \cite[Sec.~2.8]{Dy-Ka-Ku-Le}. As $X_{\mathscr{B}}$ is hereditary by definition, we have
$$
X_\eta \subset \widetilde{X}_\eta\subset X_{\mathscr{B}},
$$
where the above inclusions are strict in general, e.g.\ Example 2.42 in \cite{Dy-Ka-Ku-Le}.  Let us consider the setting $X_\eta=X_{\mathscr{B}}$.
The simplest example is $\mathscr{B}=\{1\}$, for which all statements about $X_\eta=\{0^{\mathbb Z}\}$ reduce to triviality.  In fact the equality $X_\eta=X_{\mathscr{B}}$ implies that $\mathscr{B}=\{1\}$ or that $\mathscr{B}$ is infinite \cite{Ka-Le}. As these two possibilities have to be treated separately, the former case is often excluded for convenience. The equality $X_\eta=X_{\mathscr{B}}$ also implies that $\mathscr B$ is coprime \cite{Ka-Le}.  Hence a substantial subclass is given by the Erdös case \cite[Cor.~4.2]{Ab-Le-Ru}.  As positive topological entropy of the $\mathscr{B}$-free subshift is equivalent to $\mathscr{B}$ being thin \cite{Ab-Le-Ru}, positive entropy characterises the Erdös case. Moreover, in the zero entropy case, there exist coprime Behrend sets for which the above equality holds \cite{Ka-Le}.%
\footnote{In general, for Behrend sets, the above equality need not hold. Indeed,  $X_{\mathscr{B}}$ is always uncountable \cite[Rem~2.41]{Dy-Ka-Ku-Le}, while $X_\eta$ is often countable (take, for example, $\mathscr{B}$ equal to the primes).}



\section{\texorpdfstring{$\mathscr{B}$}{B}-admissible subshifts and the Moore property}\label{sec:moore}


\subsection{Monotonicity in \texorpdfstring{\( \Bset \)}{B}-admissible subshifts}

\begin{Lemma}
\label{lemma:ShiftedCopies}
Assume that \( \Bset \) is coprime and infinite. Let \( u \in \{ 0 ,1  \}^{ \{ k , \hdots , k + \Length{u} -1 \} }  \), with \( k \in \ZZ \), be a \( \Bset \)-admissible word. Let \( N_{u} \DefAs \Card{ \Supp( u ) } \) denote the cardinality of its support. Let \( b_{0} \in \Bset \) be such that \( b_{0} > 2 \Length{u} \). Then there exist a \( \Bset \)-admissible word \( w_{u} \) and a set \( M_{u} = \{ m_{1} , \hdots , m_{ b_0 - N_{u} } \} \) of positions in \( w_{u} \) with \( \Restr{ w_{u} }{ m_{i}+k }{ m_{i} + k+\Length{u}-1 } = u \) for every \( m_{i} \in M_{u} \), such that \( M_{u} \bmod b_{0} = \{ 0 , \hdots , b_{0}-1 \} \setminus (-\Supp( u ) \bmod b_{0}) \) holds. Moreover, we have \( \Supp( w_{u} ) \bmod b_{0} = \{ 1 , \hdots , b_{0}-1 \} \) and \( \Card{ \Supp( w_{u} ) } \leq ( b_{0} - N_{u} ) N_{u} \).
\end{Lemma}

\begin{proof}
First we note that, because of \( b_{0} > \Length{u} \), the elements of \( \Supp( u ) \subseteq \{ k , \hdots , k+\Length{u}-1 \} \) are pairwise different modulo \( b_{0} \). Moreover, \( b_{0} \) is coprime to \( P \DefAs \prod_{b \in \Bset \, \setminus \{ b_{0} \} ,\, b \leq L} b \), where we choose $L\geq N_{u}( b_{0} - N_{u} )$ so that $P>|u|$. We define
\[ M_{u} = \{ m_{1} , \hdots , m_{ b_{0} - N_{u} } \} \subseteq \{ n  P :  n = 0 , \hdots , b_{0}-1 \} \]
to be those \( ( b_{0} - N_{u} ) \)-many numbers that satisfy \( m_{i} \bmod b_{0} \notin ( -\Supp( u ) \bmod b_{0} ) \). Consequently, none of the sets \( m_{i} + \Supp( u ) \) intersects the zero residue class modulo \( b_{0} \). Now, we define \( w_{u} \) as the word that is given by \( \Supp( w_{u} ) \DefAs \bigcup\nolimits_{i=1}^{b_{0}-N_{u}} ( m_{i} + \Supp( u ) )  \subseteq \{ k , \hdots , (b_{0}-1)P + k + \Length{u} - 1 \} \). In particular, \( w_{u} \) consists of \(  b_{0} - N_{u}  \) shifted copies of \( u \) at positions \( m_{i} \), which lie in all residue classes modulo \( b_{0} \) except \( -\Supp( u ) \). This implies that \( \Card{ \Supp( w_{u} ) } \leq ( b_{0} - N_{u} ) N_{u} \). Now, we check that \( w_{u} \) is \( \Bset \)-admissible:
\begin{itemize}
\item{Since none of the sets \( m_{i} + \Supp( u ) \) intersects the zero residue class modulo \( b_{0} \), also \( \Supp( w_{u} ) \) is not intersecting the zero residue class modulo \( b_{0} \).}
\item{Let \( b \in \Bset \, \setminus \{ b_{0} \} \) with \( b \leq L \). Since \( u \) is \( \Bset \)-admissible, there is a residue class modulo \( b \) that \( \Supp( u ) \) does not intersect. Since all positions of \( \Supp( u ) \) in \( \Supp( w_{u} ) \) differ by a multiple of \( b \), there is a residue class modulo \( b \) that \( \Supp( w_{u} ) \) does not intersect.}
\item{Finally, let \( b > L \). As mentioned above, \( \Card{ \Supp( w_{u} ) } \leq ( b_{0} - N_{u} ) N_{u} < b \), so \( \Supp( w_{u} ) \) cannot intersect all residue classes modulo \( b \).}
\end{itemize} 
To show the remaining assertion, i.e.\ that \( \Supp( w_{u} ) \) intersects all residue classes modulo \( b_{0} \) except the zero class, note that for every \( j \in \Supp( u ) \), the set $j+M_{u}$ intersects by construction all residue classes except those in \( j - \Supp( u ) \). Therefore, if \( \Card{ \Supp( u ) } = 1 \), we are done (as $j-\Supp(u)=\{0\}$). For \( \Card{ \Supp( u ) } > 1 \), let \( j_{1} \) and \( j_{N_{u}} \) denote respectively the smallest and the largest element of \( \Supp( u ) \). Note that all elements in \( j_{1} - \Supp( u ) \) are non-positive, while all elements in \( j_{N_{u}} - \Supp( u ) \) are non-negative. Because of \( b_{0} > 2 \Length{u} \), the zero residue class is the only residue class that all sets \( \bigcup\nolimits_{i=1}^{b_{0}-N_{u}} ( j + m_{i} ) \) miss.
\end{proof}

\begin{Prop}
\label{prop:AutomMonoton}
Assume that \( \Bset \) is coprime and infinite. Let \( \tau \colon \BadmSub \to \BadmSub \) be a CA, given by the code \( \BlCode \colon \Alphab^{ \{ -\Wind , \hdots , \Wind \} } \to \Alphab \). Then there exists \( t \in \{ -\Wind , \hdots , \Wind \} \) such that \( \tau( x ) \leq \Shift^{t} x  \) holds for all \( x \in \BadmSub \).
\end{Prop}

\begin{proof}
If \( \BlCode \) maps every word of length \( 2 \Wind +1 \) to \( 0 \), then \( \tau( x ) = 0^\Z \) for every \( x \in \BadmSub \) and the claim is trivially true. Thus, we may assume that the set \( U \DefAs \BlCode^{-1}( 1 ) \subseteq \Alphab^{ \{ -\Wind , \hdots , \Wind \} } \) is non-empty. For every \( u \in U \), we write \( J_{u} \DefAs \Supp( u ) \subseteq \{ -\Wind , \hdots , \Wind \} \) for the support and \( N_{u} \DefAs \Card{ J_{u} } \) for its cardinality.

We will prove \( \bigcap\nolimits_{ u \in U } J_{u} \neq \varnothing \). Then the assertion follows easily: let \( t \in \bigcap\nolimits_{ u \in U } J_{u} \). If \( \tau( x )( j ) = 1 \) for some \( x \in \BadmSub \) and \( j \in \ZZ \), then \( \Restr{ x }{ j - \Wind }{ j + \Wind } \in U \) holds, which implies \( x( j+t) = 1 \). Thus, we obtain \( \tau( x ) \leq \Shift^{t} x \).

To prove that the intersection is non-empty, we proceed by contradiction, so let us assume \( \bigcap\nolimits_{ u \in U } J_{u} = \varnothing \). We fix \( b_{0} \in \Bset \) with \( b_{0} > 4 \Wind +2 \). Thus, \( b_{0} \) is large enough that we can apply Lemma~\ref{lemma:ShiftedCopies} to words \( u \in \{ 0 , 1 \}^{\{ - \Wind , \hdots , \Wind \}} \), that is, with \( k=-\ell \) and \( \Length{ u } = 2 \Wind +1 \). For every \( u \in U \), let \( w_{u} \) denote the \( \Bset \)-admissible word constructed in Lemma~\ref{lemma:ShiftedCopies}, which contains \( b_{0} - N_{u} \) occurrences of \( u \), starting in pairwise different residue classes modulo \( b_{0} \). Now, we arrange all the \( w_{u} \)'s in a single element \( x \in \BadmSub \). Then we show that the support of \( \tau( x ) \) intersects every residue class modulo \( b_{0} \). Hence \( \tau( x ) \) is not \( \Bset \)-admissible, which yields a contradiction.

We write \( \Bset \, \setminus \{ b_{0} \} = \{ b_{1} , b_{2} , b_{3} , \hdots \} \), where the elements \( b_{k} \) are in increasing order. Let \( K \) be such that \( b_{K+1} > \max( \Card{U} \cdot ( b_{0} - (2\ell+1))(2\ell+1) , b_{0} ) \) holds. For every \( u \in U \), the associated word \( w_{ u } \) from Lemma~\ref{lemma:ShiftedCopies} with \( \Supp( w_{u} )  \subseteq \{ -\Wind , \hdots , (b_{0}-1)P + \Wind \} \)  is \( \Bset \)-admissible. Hence, for every \( b_{i} \in \Bset \) there is at least one residue class \( r_{u, i} \) modulo \( b_{i} \) that \( \Supp( w_{u} ) \) misses. By the Chinese Remainder Theorem, there exists \( s_{u} \in \NN \) which solves \( s_{u} + r_{u, i} \equiv 0 \bmod b_{i} \) simultaneously for $i=0,1,\ldots,K$. If we define \( \Restr{ x }{ k_{u} \cdot \prod_{i=0}^{K} b_{i} + s_{u} - \Wind }{ k_{u} \cdot \prod_{i=0}^{K} b_{i} + s_{u} - \Wind + \Length{ w_{u} } - 1 } \DefAs w_{u} \) for some \( k_{u} \in \NN \), then the support of \( x \) in this interval misses the zero residue class modulo \( b_{i} \) for all \( i \in \{ 0 , \hdots , K \} \). We do this for every \( u \in U \), so that \( x \) contains an occurrence of each \( w_{u} \). This yields that \( \Card{ \Supp( x ) } \leq \Card{ U } \cdot ( b_{0} - N_{u} )N_{u} \) and since \( \frac{ b_{0} }{ 2 } > 2\ell + 1 \geq N_{u} \), this is bounded from above by \( \Card{ U } \cdot ( b_{0} - (2\ell + 1) )(2\ell + 1) \). Hence, the sequence \( x \) is \( \Bset \)-admissible, since its support misses the zero residue class for \( b_{0} , \hdots , b_{K} \) by construction and misses some residue class for each \( b \geq b_{K+1} > \Card{U} \cdot ( b_{0} - (2\ell+1))(2\ell+1) \) due to having at most \( \Card{U} \cdot ( b_{0} - (2\ell+1)) (2\ell+1) \) elements.

By the construction in Lemma~\ref{lemma:ShiftedCopies}, \( w_{ u } \) contains copies of \( u \) at \( \{ m_{i}-\Wind , \hdots , m_{i} + \Wind \} \), that is,  copies of \( u \) centered at \( \{ m_{u, 1} , \hdots , m_{u, b_{0}-N_{u}} \} \AsDef M_{u} \) with \( M_{u} \bmod b_{0} = \{ 0 , \hdots , b_{0}-1 \} \setminus (-J_{u}) \). Moreover, \( \Supp( w_{u} ) \) misses only the zero residue class modulo \( b_{0} \), which implies \( s_{u} \equiv 0 \bmod b_{0} \). Since \( x \) contains \( w_{u} \) starting at position \( k_{u} \cdot \prod_{i=0}^{K} b_{i} + s_{u} - \Wind \), it contains \( u \) centered at positions \( k_{u} \cdot \prod_{i=0}^{K} b_{i} + s_{u} + m_{u, i} \), that is, in the residue classes \( M_{u} \bmod b_{0}\). Because of \( u \in U \), it follows that \( \tau( x ) \) has (mod~$b_0$) \( 1 \)'s in \( \{ 0 , \hdots , b_{0}-1 \} \setminus (-J_{u}) \). Since we assumed \( \bigcap\nolimits_{ u \in U } J_{u} = \varnothing \) and \( b_{0} > 2 \Wind +1 \), the support of \( \tau( x ) \) intersects every residue class modulo \( b_{0} \).
\end{proof}

\subsection{Proof of the Moore property}

\begin{Prop}\label{prop:Daniel}
Assume that \( \BfreeSub = \BadmSub \). Let \( \tau \colon \BfreeSub \to \BfreeSub \) be a surjective CA, given by the code \( \BlCode \colon \Alphab^{ \{ -\Wind , \hdots , \Wind \} } \to \Alphab \). Then there exists \( t \in \{ -\Wind , \hdots , \Wind \} \) with \( \tau = \Shift^{t} \). In particular, $(X_\eta,S)$ satisfies the Moore property.
\end{Prop}

\begin{proof}
Since \( \tau \) is surjective, there exists \( z \in \BfreeSub \) with \( \tau( z ) = \Bseq \).  Since $X_\eta=X_{\mathscr{B}}$, the set $\mathscr{B}$ is coprime by \cite{Ka-Le}. Moreover, \( \Bset \) is infinite, since otherwise the periodicity of \( \Bseq \) would imply a finite set \( \BfreeSub \), while \( \BadmSub \) is always uncountable. Hence Proposition~\ref{prop:AutomMonoton} yields \( \Bseq = \tau( z ) \leq \Shift^{t} z \) for some \( t \in \{ -\ell , \hdots , \ell \} \). Since \( \Bseq \) is a maximal element in \( \BfreeSub \) by Lemma~\ref{maksymalnosc}, we obtain \( \Bseq = \Shift^{t} z \). Moreover, \( \tau \) commutes with the shift, which yields \( \tau( \Bseq ) = \tau( \Shift^{t} z ) = \Shift^{t} \tau( z ) = \Shift^{t} \Bseq \). Since $\eta$ is a transitive point, the result follows.
\end{proof}

\begin{Remark} The above corollary applies to the Erdös case, where  the equality $X_\eta=X_{\mathscr{B}}$ has been established \cite[Cor.~4.2]{Ab-Le-Ru}. But there are also coprime Behrend sets for which we have this equality \cite{Ka-Le}.
\end{Remark}

\section{Hereditary subshifts and the Myhill property}\label{sec:myhill}

\subsection{A discussion of pre-injectivity}

The following weakening of injectivity due to Gromov is standard in the theory of cellular automata \cite{Ce-Co}.

\begin{Def}
A CA $\tau$ on a subshift $(X,S)$ is \textit{pre-injective} if, for any $x,y\in X$, the property $\tau(x)=\tau(y)$ implies that either $x=y$ or that $x$ differs from $y$ at infinitely many coordinates.
\end{Def}

In order to get acquainted with the notion of pre-injectivity, consider the following two auxiliary results. The first one addresses the interplay between (pre-)injectivity and monotonicity.

\begin{Lemma}\label{lem:X0}
Consider some hereditary set $X_0\subset \{0,1\}^\Z$ of $0-1$-sequences having finite support.  Assume that $\tau: X_0\to \{0,1\}^\Z$ is injective and monotone, i.e., $\tau(x)\le x$ coordinatewise for any $x\in X_0$. Then $\tau$ is the identity map.
\end{Lemma}

\begin{proof}
As $X_0$ is hereditary and as $\tau$ is monotone, we have $\tau(X_0)\subset X_0$. We can thus consider iterates of $\tau$. Consider any $x\in X_0$. As $x$ has at most finitely many nonzero coordinates and as $\tau$ is monotone, there exists $n\in \mathbb N$ such that  $\tau^{n+1}(x)=\tau^n(x)$. Now, the  injectivity of $\tau$ yields $\tau(x)=x$.
\end{proof}

The next result is probably well-known. As we were not able to trace a reference which applies to our setting, we also provide its proof. We say that $\tau:X\to Y$ is bounded-to-one if there exists a finite constant $B$ such that $|\tau^{-1}(\{y\})|\le B$ for all $y\in Y$.

\begin{Lemma}\label{l:goe1}
Let $(X,S)$ be a hereditary subshift with a CA $\tau$. If $\tau$ is pre-injective, then $\tau$ is bounded-to-one.
\end{Lemma}

\begin{proof}
Recall that if we consider distinct $x,y\in X$ such that $\tau(x)=\tau(y)$ then by pre-injectivity assumption, $x$ and $y$ differ at infinitely many positions.

\begin{quote}
We claim that there exists a universal constant $\delta>0$ independent of $x,y$ with the following property: If $x$ and $y$ differ at infinitely many positive positions, we have $d(S^{r}x,S^{r}y)\geq \delta$ for all but finitely many $r\in\mathbb N$. If $x$ and $y$ differ at infinitely many negative positions, we have $d(S^{-r}x,S^{-r}y)\geq \delta$ for all but finitely many $r\in\mathbb N$.
\end{quote}

\noindent The above claim implies the statement of the lemma by the following argument. Fix any $x\in X$ and assume that  $\tau^{-1}(\{x\})$ contains $N$ elements  $\{x_1, \ldots, x_N\}$. As $X$ is compact, we find $\{z_1^+, z_1^-,\ldots, z_N^+, z_N^-\}\subset X$ and an increasing sequence $(r_k)_{k\in \mathbb N}$  of natural numbers such that $S^{r_k}x_i\to z_i^+$ and $S^{-r_k}x_i\to z_i^-$ for all $i$. This may be seen by an iterative argument taking subsequences. Now, for any  $i\neq j$ the claim implies $d(z_i^+, z_j^+)\ge \delta$ or $d(z_i^-, z_j^-)\ge \delta$, by the
continuity of the metric. Thus by the compactness of $X$, the number $N$ cannot be arbitrarily large and is bounded uniformly in $X$. Hence $|\tau^{-1}(\{x\})|$ can be bounded by some finite number that does not depend on the choice of $x$.

\smallskip

\noindent  As the map $\tau$ is a CA, it can be described by a code of length $\ell\ge0$ as in \eqref{clh}.
Choose  $\delta_\ell>0$ sufficiently small such that for $x,y\in X$ we have  $x(-4\ell,4\ell)=y(-4\ell,4\ell)$ whenever $d(x,y)<\delta_{\ell}$. We prove the above claim for $\delta=\delta_\ell$ by an indirect argument. Consider distinct $x,y\in X$ such that $\tau(x)=\tau(y)$. Assume without loss of generality that $x$ and $y$ differ at infinitely many positive positions. In order to argue for contradiction, assume that there exists an increasing sequence of natural numbers $(r_k)_{k\in \mathbb N}$ such that $d(S^{r_k}x,S^{r_k}y)<\delta_\ell$ for every $k\in\mathbb N$. By our choice of $\delta_\ell$, we then have
\beq\label{eq:equal}
x(r_k-4\ell,r_k+4\ell)=y(r_k-4\ell,r_k+4\ell)
\eeq
for every $k\in \mathbb N$. As $x$ and $y$ differ at infinitely many positive positions, we may select $k_0<k_1$ such that
\beq\label{eq:unequal}
x(r_{k_0}-4\ell,r_{k_1}+4\ell)\neq y(r_{k_0}-4\ell,r_{k_1}+4\ell).
\eeq
Let $D=\{r_{k_0}-4\ell,\ldots, r_{k_1}+4\ell\}$ and define $x',y'\in \{0,1\}^{\Z}$ by $x'=x\cdot \raz_D$ and $y'=y\cdot \raz_D$, the dot denoting coordinatewise multiplication. Then $x'$ and $y'$ differ due to \eqref{eq:unequal}, but they differ at finitely many positions only. Moreover, $x', y'\in X$ by heredity.  On the other hand we have $\tau(x')=\tau(y')$, which contradicts the pre-injectivity of $\tau$. To see the latter, consider $n\in \{r_{k_0}-2\ell, \ldots r_{k_1}+2\ell\}$ and note $\tau(x')(n)=\tau(x)(n)$ and $\tau(y')(n)=\tau(y)(n)$ as the code has length $\ell$. Thus $\tau(x')(n)=\tau(y')(n)$ due to $\tau(x)=\tau(y)$. For $n\notin \{r_{k_0}-2\ell, \ldots, r_{k_1}+2\ell\}$, note that $\tau(x')(n)$ and $\tau(y')(n)$ are both coded by the same interval by \eqref{eq:equal}. Hence $\tau(x')(n)=\tau(y')(n)$. Altogether, we obtain $\tau(x')=\tau(y')$.
\end{proof}

\subsection{Myhill property on classes of hereditary subshifts}

The above two lemmata are crucial ingredients for a proof of the Myhill property. Our first result is based on Lemma~\ref{lem:X0} and uses monotonicity.

\begin{Prop}\label{prop:Myhill1}
If $X_\eta=X_{\mathscr{B}}$, the $\mathscr{B}$-free subshift $(X_\eta,S)$ satisfies the Myhill property.
\end{Prop}

\begin{proof}
Assume that $\tau$ is a pre-injective CA on $X_\eta$. As $X_\eta=X_{\mathscr{B}}$, the set $\mathscr{B}$ is infinite and coprime  by \cite{Ka-Le}. Hence $\tau\circ S^{t}$ is monotone by Proposition~\ref{prop:AutomMonoton} for some integer $t$. Also note that $\tau\circ S^{t}$ inherits pre-injectivity from $\tau$.
Denote by $X_0\subset X_\mathscr{B}$ the collection of all $0-1$-sequences in $X_\mathscr{B}$ of finite support. Then $\tau\circ S^{t}$ is the identity on $X_0$ by  Lemma~\ref{lem:X0}. Finally, note that $X_0$ is dense in $X_\mathscr{B}$.  As $X_\eta=X_{\mathscr{B}}$, it is also dense in $X_\eta$.  By the continuity of $\tau$ we thus get that $\tau\circ S^{t}$ is the identity on $X_\eta$. In particular, $\tau$ is surjective.%
\footnote{Recall that $X_\eta=X_{\mathscr{B}}$ also for $\mathscr{B}=\{1\}$, in which case the assertion holds trivially.}
\end{proof}

The previous result applies in particular to the Erdös case. We have the following version for hereditary subshifts, where we use Lemma~\ref{l:goe1}, which does not resort to monotonicity.

\begin{Prop}\label{t:goe1}
Let $(X,S)$ be a hereditary subshift which is intrinsically ergodic and whose measure of maximal entropy has full topological support. Then $(X,S)$ has the Myhill property.
\end{Prop}

\begin{proof}
Let $\tau$ be a pre-injective CA on $(X,S)$ and suppose $Y:=\tau(X)\subset X$ with $Y\neq X$. Then $(Y,S)$ is a subsystem of $(X,S)$ satisfying $h(Y,S)=h(X,S)$. The latter statement uses the variational principle and the fact that all fibers $\tau^{-1}(\{x\})$ have finite cardinality by Lemma~\ref{l:goe1}. As $(Y,S)$ is a subshift, we may take a measure $\nu$ on $Y$ of maximal entropy. Then $\nu$ is also a measure of maximal entropy for $(X,S)$. But by assumption $\nu$ has full support on $X$, which is in conflict with $Y\neq X$.
\end{proof}

Now, notice that if $\nu_\eta$ has full support, then also $\nu_\eta\ast B(\frac12,\frac12)$ has full support in $\widetilde{X}_\eta$ (cf.\ \cite{Ku-Le(jr)}) and therefore, in view of Proposition~\ref{t:goe1} and Theorems~\ref{atw4} and~\ref{atw1}, we obtain the following.

\begin{Cor}\label{c:goe1}
The Myhill property holds for $(\widetilde{X}_\eta,S)$ whenever $\mathscr{B}$ is taut. In particular, it holds in the Erdös case. \qed
\end{Cor}

\begin{Remark} Apart from the Erdös case, there are more examples of hereditary systems for which the assumptions of Proposition~\ref{t:goe1} are satisfied. For instance, they are satisfied for the hereditary closures of Sturmian sequences \cite[Sec.~4]{Ku-Le-We}. In particular, the Myhill property holds for such closures.
\end{Remark}


\subsection{GoE theorem for \texorpdfstring{$X_\eta=X_\mathscr{B}$}{X\_eta = X\_B}}

The following result contains Theorem~\ref{t:main} from the introduction as a special case. Indeed, in the Erdös case we have $X_\eta=X_\mathscr{B}$, see \cite[Cor.~4.2]{Ab-Le-Ru}.\footnote{It has been proved in \cite{Ka-Le} that there are non-thin subsets $\mathscr{B}$ of primes numbers (such a $\mathscr{B}$ is Behrend) such that we also have $X_\eta=X_{\mathscr{B}}$.}

\begin{Th}\label{t:GoEforBehrend}
Assume that a $\mathscr B$-free subshift $(X_\eta,S)$ satisfies $X_\eta=X_\mathscr{B}$.  Then it satisfies the GoE theorem.  In fact, every CA is given by a monotone code, modulo a power of the shift. If the CA is onto, then it equals a power of the shift.
\end{Th}

\begin{proof}
The Moore property has been proved in Proposition~\ref{prop:Daniel}. The Myhill property and the remaining statements follow from Proposition~\ref{prop:Myhill1}.
\end{proof}

\section{The automorphism group is trivial for hereditary \texorpdfstring{$\mathscr{B}$}{B}-free systems} \label{s:mentzen}

\subsection{Some facts about hereditary subshifts}

Assume that $X\subset \{0,1\}^{\Z}$ is a hereditary subshift.
Assume additionally that $(X,S)$ is transitive and let the orbit $(S^nz)_{n\in\Z}$ of $z\in X$ be dense in $X$. We will constantly assume that $z$ has infinite support (otherwise the dynamics of $(X,S)$ is very simple: it consists of finitely many orbits attracted by a fixed point).
Given a non-negative $a\in\Z$ and a block $B \in \{0,1\}^{\{-a,\ldots,a\}}$, we say that this block {\em appears centrally in} $z$ {\em at position} $m$ if $z(m+n)=B(n)$ for each $n\in [-a,a]$. It follows that the set of those $m$ such that $B$ appears centrally in $z$ at position $m$ must be infinite. Indeed, let $z(m+n)=B(n)$ for each $n\in [-a,a]$. Since $z$ has infinite support, there exists $j\geq1$ such that for some $s\in[-(a+j),a+j]\setminus[-a,a]$, we have $z(m+s)=1$. Now define a new block $C\in\{0,1\}^{\{-(a+j),\ldots,a+j\}}$ arising from $z(m-(a+j),m+a+j)$  by replacing all the 1's at positions in $[m-(a+j),m+a+j]\setminus[-a,a]$ by the 0's. By the heredity of $X$, it follows that the block $C$ has to appear (centrally) in $z$ at a certain position $m'$ with necessarily $m'\neq m$. Now, repeat the same reasoning with $B$ replaced by $C$, etc. In this way, we obtain an infinite sequence of different positions of (central) appearances of the original block $B$ in $z$.

\smallskip

Assume additionally that the  point $z\in X$ is  maximal, that is: for each $x\in X$, if $z\leq x$ then $z=x$.  Note that whenever $z$ is a transitive maximal point and $\tau\in {\rm Aut}(X,S)$ satisfies $\tau(x)\leq x$ for each $x\in X$ then $\tau=Id$. Indeed, $z\leq \tau^{-1}(z)$, which means that $\tau(z)=z$ and $\tau=Id$ follows from the fact that $z$ is a transitive point.

\smallskip

Note also that the all zero sequence $0^\Z$  is a fixed point belonging to $X$. Unless $(X,S)$ is the full shift, $0^\Z$ is the only fixed point for $S$. Hence, assuming that $(X,S)$ is not the full shift, we must have
\beq\label{e2}
\tau( 0^\Z ) = 0^\Z\text{ for each }\tau\in {\rm Aut}(X,S).
\eeq
Given $\tau\in {\rm Aut}(X,S)$, let  $\phi:\{0,1\}^{\{-\ell,\ldots,\ell\}}\to\{0,1\}$ be the corresponding code
\beq\label{e3}
\tau(x)(n)=\phi(x(n-\ell,n+\ell))\text{ for each }n\in\Z.
\eeq
By an abuse of notation, we will also write that $\tau(x)=\phi(x)$ (or even $\phi(C)$, where $C$ is a block).
It follows from \eqref{e2} and~\eqref{e3} that $\phi(0^{\{-\ell,\ldots,\ell\}})=0$.

\subsection{The result and its proof}

The following result extends Mentzen's theorem beyond the Erdös case. It also extends the result from Theorem~\ref{t:GoEforBehrend} (for the invertible CAs) in the hereditary case $X_\eta=X_\mathscr{B}$.

\begin{Th}\label{t1}
Let $(X_\eta, S)$ be a $\mathscr B$-free subshift. If $X_\eta$ is hereditary then ${\rm Aut}(X_\eta,S)$ is trivial, that is, it equals $\{S^n:\:n\in\Z\}$.
\end{Th}
Hence, in view of Theorem~\ref{atw2}, the following holds:
\begin{Cor}\label{c:t1}
Consider a proximal $\mathscr B$-free subshift $(X_\eta, S)$ such that $\mathscr{B}$ is taut. Then ${\rm Aut}(X_\eta,S)$ is trivial. \qed
\end{Cor}

\begin{Remark}
The result contained in Corollary~\ref{c:t1} was already proved in \cite[Thm.~4.1]{Ke-2020} using so called generic dynamics tools.
\end{Remark}

To prove Theorem~\ref{t1}, we will need two lemmata which follow the strategy of the proof of Mentzen's theorem \cite{Me}.
Let $u\in\{0,1\}^{\Z}$ be defined as $u(n)=0$ for each $n\in\Z$ except for $n=0$. We write $u=0^{-\N}10^\N$. By heredity, $u\in X_\eta$.

\begin{Lemma}\label{l1} There exists $k\in\Z$ such that $\tau( u ) = S^ku$. Equivalently, the code $\phi:\{0,1\}^{\{-\ell,\ldots,\ell\}}\to\{0,1\}$ has the property that
it takes value~1 on exactly one block of length $2\ell+1$ having one element support.
\end{Lemma}
\begin{proof}
Note that $\phi$ cannot take value~0 on all blocks of length $2\ell+1$ with one element support as otherwise $\tau( u )=0^\Z$ and $\tau$ would not be invertible. Suppose  that the support of $\tau( u ) $ has at least two elements. So let $C_1,C_2$ be two different blocks of length $2\ell+1$ having one-element support and $\phi(C_1)=\phi(C_2)=1$. Let
$C_1(a_1)=C_2(a_2)=1$ and set $s:=a_1-a_2$. Choose $b\in\mathscr{B}$ so that
\beq\label{e7}
\gcd(s,b)=1.
\eeq
Such a $b$ exists as $X_\eta$ is hereditary, whence $(X_\eta,S)$ is proximal and therefore $\mathscr{B}$ contains an infinite coprime subset.
Let, with $r$ maximal,  $m_1,\ldots,m_{r}$ be such that $\eta(m_i)=1$ for $i=1,\ldots,r$ and $m_i$ mod $b$ are pairwise different (we recall that on $\eta$ at least the zero residue class of $b$ does not appear, so $r<b$ and all $m_{i}$ are different from zero modulo $b$). We claim that, with no loss of generality, we can assume that the distances $|m_i-m_j|$ for $i\neq j$ are as large as we want. Indeed, as we have seen, the heredity of $X_\eta$ implies that $\eta$ is recurrent. This means that for infinitely many $n\in\Z$ the elements $m_1+n,\ldots,m_{r}+n$ satisfy the aforementioned properties of $m_1,\ldots,m_{r}$ (the sets $\{m_1,\ldots,m_{r}\}$ and $\{m_1+n,\ldots,m_{r}+n\}$ must be equal modulo $b$ by the maximality of $r$). Now, proceed in the following way: take such a ``good'' $n_2$ arbitrary large and among elements
$m_1+n_2,\ldots,m_{r}+n_2$ find a one, say $m_j+n_2$ such that
$$
m_j+n_2=m_2\text{ mod }b$$
and replace $m_2$ by $m_j+n_2$. In this way $m_1$ and $m_2$ are well separated. Then choose another ``good'' $n_3$ (arbitrary large) so that
$$
m_j+n_3=m_3\text{ mod }b$$
for some $j$ and replace $m_3$ by $m_j+n_3$. We will obtain our claim in $r$ steps.

Having the elements $m_1,\ldots, m_{r}$ well separated, we consider a large interval $I$ in $\Z$ containing all these elements. Then let $C$ be the block whose positions are indexed by $I$ so that the only positions of~1's are given by $m_1,\ldots, m_{r}$. By heredity (and the fact that $\eta$ is transitive), it follows that the block $C$ appears in $\eta$ at a certain position $n$. Now $\tau( \eta ) \in X_\eta$, and since \( \tau \) is given by the code \( \phi \), it follows that \( \tau( \eta ) \) contains the image \( \phi(C) \) of the block $C$. Thus, for some $n'\in\Z$, we have
$$
\tau( \eta ) ( m_i+n' )=1=\tau( \eta )(m_i+n'+s)\text{ for }i=1,\ldots,r.$$
It follows that the set $Z$ consisting of $m_i+n'$ mod $b$, $i=1,\ldots,r$, is a subset contained in the support of $\eta$ modulo $b$ and so is the set $Z+s$ mod $b$. Since the cardinality of $Z$ is maximal, $Z+s=Z$ which is in conflict with~\eqref{e7}.
\end{proof}

In what follows, we replace $\tau\in {\rm Aut}(X_\eta,S)$ by $S^{-k}\tau$, and, by slightly enlarging the length $\ell$ if necessary, we can assume that the only block of length $2\ell+1$ with one element support whose code is 1 is the block that has 1 at the central position.  We continue the proof of Theorem~\ref{t1}, but now we have to show that $\tau=Id$.

\begin{Lemma}\label{l2} For an arbitrary block $C$ of length $2\ell+1$ with the zero at the central position, we have $\phi(C)=0$.\end{Lemma}
\begin{proof}
Suppose that for some $C$ as above, $\phi(C)=1$. Let $C$ appear (centrally) in $\eta$ at position $s$. We have $\eta(s)=0$, so there exists $b\in\mathscr{B}$ such that $s=mb$ for some $m\in\Z$. Choose positions $m_1,\ldots,m_{r}$ with \( \eta( m_{i} ) = 1 \) so that the distances between these elements are as large as we want (by the proof of Lemma~\ref{l1}), and mod $b$ they ``realize'' the maximal number $r$ of non-zero residue classes in $\eta$. In particular, the $m_{i}$ are all different from $s$.  Take a large block $E$ containing the following positions: $s$ (with the whole block~$C$) and $m_1,\ldots,m_{r}$. Replace all 1's that appear in it at all other positions than those given by $C$ and $m_1,\ldots, m_{r}$ by 0's, to obtain the block $F$.
By heredity, the block $F$ appears in $\eta$ at a certain position $r'$. It easily follows that the support of $\tau( \eta )$ will now contain a certain translation of the set $\{s,m_1,\ldots,m_{r}\}$, and since this set mod $b$ is of cardinality $\geq r+1$, we obtain $\tau( \eta ) \notin X_\eta$ and hence a contradiction.
\end{proof}

\noindent
{\bf Proof of Theorem~\ref{t1}} It follows from Lemma~\ref{l2} that for each $x\in X_\eta$, we have $\tau( x ) \leq x$.
It follows that $\eta\leq \tau^{-1}( \eta )$, and since Lemma~\ref{maksymalnosc} holds
in $X_\eta$, the result follows. \bez


\section{Proximal and taut \texorpdfstring{$\mathscr{B}$}{B}-free subshifts}\label{sec:hertaut}


Now, let $\mathscr{B}$ be taut and and assume that
the $\mathscr{B}$-free subshift $(X_\eta, S)$ is proximal. In view of Sections~\ref{sec:taut} and~\ref{sec:probeh}, the tautness of $\mathscr{B}$ has the following consequences: proximality and heredity of $(X_\eta,S)$ are equivalent notions, the Mirsky measure has full topological support $X_\eta$, and hence also the unique measure of maximal entropy $m_{\max}\in M(X_\eta,S)$ has full topological support $X_\eta$. In this section, we will prove the following theorem.\footnote{It holds trivially when $\mathscr{B}=\{1\}$.}

\begin{Th}\label{Thm:tautprox}
Suppose that $\mathscr{B}$ is taut and that the associated $\mathscr{B}$-free subshift $(X_\eta,S)$ is proximal. Let $\tau$ be any CA of $(X_\eta,S)$. Then the following are equivalent:
\begin{itemize}
\item[(a)] $m_{\max}(\tau(X_\eta))>0$.
\item[(b)] $\text{int}_{X_\eta}(\tau(X_\eta))\ne\varnothing$.
\item[(c)] $\tau$ is surjective.
\item[(d)] $\tau$ is pre-injective.
\item[(e)] $\tau$ is a power of the shift on $X_\eta$.
\end{itemize}
In particular,  $(X_\eta,S)$ satisfies the GoE theorem.
\end{Th}

\begin{Remark}
Theorem~\ref{Thm:tautprox} complements Theorem~\ref{t:GoEforBehrend}. The proof of Theorem~\ref{Thm:tautprox} makes heavy use of the property that the measure of maximal entropy has full topological support. So it will never apply to Behrend sets $\mathscr{B}$, for example, to which Theorem~\ref{t:GoEforBehrend} may apply. On the other hand, Theorem~\ref{t:GoEforBehrend} applies only to coprime $\mathscr{B}$, hence not to the $\mathscr{B}$-free sets of deficient numbers or of super-level sets of Euler's $\phi$-function. As these examples are based on taut and proximal $\mathscr{B}$, Theorem~\ref{Thm:tautprox} applies to them.
\end{Remark}


The proof of Theorem~\ref{Thm:tautprox} needs some preparation. The following lemma is certainly folklore for experts, but we could not locate it in the literature. We provide a proof for convenience.

\begin{Lemma}\label{lem:GK-1}
Let $X$ be a compact metric space, $S:X\to X$ continuous, and suppose $(X,S)$ has a unique measure $m_{\max}$ of maximal entropy. Let $\tau$ be an endomorphism of $(X,S)$ such that $m_{\max}(\tau(X))>0$. Then the following hold.
\begin{itemize}
\item[(a)] $m_{\max}(\tau(X))=1$.
\item[(b)] $m_{\max}$ is the only $S$-invariant measure $m$ on $X_\eta$ that satisfies $m\circ \tau^{-1}=m_{\max}$.
\item[(c)] If $y$ is generic for $m_{\max}$ and if $y=\tau(x)$, then $x$ is generic for $m_{\max}$.
\end{itemize}
\end{Lemma}

\begin{proof}
Observe first that $m_{\max}(\tau(X))=1$ as $m_{\max}$ is ergodic and $\tau(X)$ is $S$-invariant.  For the uniqueness claim in $(b)$, assume that $m\in M(X,S)$ satisfies $m\circ \tau^{-1}=m_{\max}$. We then have $h_{top}(S)=h_{m_{\max}}(S)\le h_m(S)\le h_{top}(S)$. Thus $m=m_{\max}$, as $(X,S)$ has a unique measure of maximal entropy. For existence, pick any generic point $y$ for $m_{\max}$ satisfying $\tau^{-1}(\{y\})\ne\varnothing$.  As $m_{\max}(\tau(X))=1$, the latter two properties are in fact true for $m_{\max}$-a.a.~$y$. Consider any $x\in \tau^{-1}(\{y\})$. By the compactness of $M(X,S)$, we may choose $m\in M(X,S)$ such that $x$ is generic for $m$ along some subsequence $(n_k)$. We then have for any $f\in C(X)$ that
\begin{displaymath}
\begin{split}
\int_X f \, {\rm d} (m\circ \tau^{-1}) &= \int_X f\circ \tau \, {\rm d} m = \lim_{k\to\infty}\frac1{n_k} \sum_{i=0}^{n_k-1} (f\circ \tau)(S^i x)\\
&= \lim_{k\to\infty}\frac1{n_k} \sum_{i=0}^{n_k-1} f(S^i y) = \int_X f \, {\rm d} m_{\max}.
\end{split}
\end{displaymath}
Hence $m$ satisfies $m\circ \tau^{-1}=m_{\max}$. As $m=m_{\max}$ by the above uniqueness argument, this shows $(b)$ and $(c)$.
\end{proof}

\begin{Lemma}\label{lem:GK-2}
Assume  that $(X,S)$ is a hereditary subshift which is intrinsically ergodic. Let $m_{\max}$ be the unique measure of maximal entropy and assume additionally that  $m_{\max}$ has full topological support. Let $\tau$ be a CA of $(X,S)$ such that $m_{\max}(\tau(X))>0$. Then $\tau$ is surjective and $\tau(0^{-\mathbb N}10^{\mathbb N})\ne 0^{\mathbb Z}$.
\end{Lemma}

\begin{proof}
As $\tau(X)$ is a compact set of full $m_{\max}$-measure by Lemma~\ref{lem:GK-1} and as $m_{\max}$ has full topological support, we have $\tau(X)=X$. The remainder of the proof is based on the Abramov-Rohlin formula, see e.g.~\cite{Le-Wa}. In our setting, its formula $(1.2)$ shows
\begin{equation}\label{eq:GK-1}
h_{m_{\max}}(S)=h_{m_{\max}}(S) + \int_X h(S,\tau^{-1}(\{y\})) \, {\rm d} m_{\max}(y),
\end{equation}
where
\begin{displaymath}
h(S,A)= \lim_{\delta\to0} \limsup_{n\to\infty} \frac{1}{n} \log\sup \left\{\mathrm{card}(E): \text{ $E\subset A$ is $(n,\delta)$-separated} \right\}
\end{displaymath}
for each $A\subset X$.
Let the CA $\tau$ be given by a block code $\phi$ of length $\ell$ and suppose for a contradiction that $\tau(0^{-\mathbb N}10^{\mathbb N})= 0^{\mathbb Z}$. Then $\phi(B)=0$ for each block $B\in \{0,1\}^{\{-\ell,\ldots,\ell\}}$ that contains at most one symbol $1$. Denote by $C$ the block $0^{2\ell}10^{2\ell}$. On all $m_{\max}$-generic $x\in X$, the block $C$ appears in $x$ with the same frequency $\kappa:= m_{\max}(C)>0$, as $X$ is hereditary and $m_{\max}$ is an ergodic measure with full topological support for $(X,S)$. At each of the occurrences of $C$ in such an $x$ we can randomly switch the $1$ in the middle of the block to $0$ or not. In any case we obtain a modified sequence $x'\in X$ such that $\tau(x')=\tau(x)$. Hence for all such $x$ we have
\begin{displaymath}
h(S, \tau^{-1}(\{\tau(x)\}))\ge \kappa\cdot \log 2.
\end{displaymath}
In view of Lemma~\ref{lem:GK-1} this shows that
\begin{displaymath}
h(S, \tau^{-1}(\{\tau(x)\}))\ge \kappa\cdot \log 2 \qquad \text{ for $m_{\max}$-a.a.\ $x$}.
\end{displaymath}
As $h_{m_{\max}}(S)<\infty$, this contradicts identity \eqref{eq:GK-1}.
\end{proof}

\begin{Lemma}\label{l:NotAllZeroMoore} Consider a hereditary $\mathscr{B}$-free subshift $(X_\eta, S)$. Assume that a CA $\tau$ on $(X_\eta,S)$ is surjective and that $\tau(0^{-\N}10^\N)\neq0^\Z$. Then $\tau$ is a power of a shift. In particular, $(X_\eta, S)$ satisfies the Moore property.
\end{Lemma}

\begin{proof}
Write $x_0=0^\Z$ and $x_1=0^{-\N}10^\N$ and note that indeed $x_0, x_1\in X_\eta$ by heredity.
Note first that there exists $r\in\Z$ such that $\tau(x_1)=S^rx_1$ by the same proof as in Lemma~\ref{l1}, where $\tau(0^{-\N}10^\N)\neq0^\Z$ is now true by assumption and not as a consequence of injectivity. As $\tau$ is surjective if and only if $\tau\circ S^{-r}$ is surjective, by redefining $\tau$ we may assume without loss of generality that
$\tau(x_1)=x_1$. Now, Lemma~\ref{l2} yields $\tau(x)\leq x$ for each $x\in X_\eta$.
As $\tau$ is assumed surjective, there exists $z\in X_\eta$ such that $\tau(z)=\eta$. But $\eta=\tau(z)\leq z$,
which is only possible if $\eta=z$ by the maximality property of $\eta$ stated in Lemma~\ref{maksymalnosc}. Since $\eta$ is a transitive point in $X_\eta$, we get that $\tau$ is the identity.
\end{proof}

We can now prove Theorem~\ref{Thm:tautprox}.

\begin{proof}[Proof of Theorem~\ref{Thm:tautprox}]

$(a) \Rightarrow (e)$: This follows from Lemma~\ref{lem:GK-2} and Lemma~\ref{l:NotAllZeroMoore}.

\noindent $(e) \Rightarrow (d)$: This is trivial.

\noindent $(d) \Rightarrow (c)$: This follows from Proposition~\ref{t:goe1}.

\noindent $(c) \Rightarrow (b)$: This is trivial.

\noindent $(b) \Rightarrow (a)$: This holds as $m_{\max}$ has support $X_\eta$.
\end{proof}

\section{Proximal \texorpdfstring{$\mathscr B$}{B}-free subshifts and open questions} \label{s:direction}

It is not clear to us to which extent Theorem~\ref{t1} is true in the (general) proximal case. We have however the following reduction of the problem.

\begin{Cor}\label{wn2}
Let $(X_\eta, S)$ be a proximal $\mathscr{B}$-free subshift. Let $\mathscr{B}'$ be the unique taut set as in Theorem~\ref{atw3}. Assume that $\tau\in {\rm Aut}(X_\eta,S)$. Then, $X_{\eta'}={\rm supp}(\nu_\eta)={\rm supp}(\nu_{\eta'})$ and $\tau( X_{\eta'} ) =X_{\eta'}$ with $\tau|_{X_{\eta'}}=S^k$ for some $k\in\Z$.
\end{Cor}

\begin{proof} Since $\mathscr{B}'$ is taut, by Theorem~\ref{atw1}, $X_{\eta'}={\rm supp}(\nu_{\eta'})={\rm supp}(\nu_\eta)$. Moreover, by Corollary~\ref{c:entfor}, $(X_\eta,S)$ is intrinsically ergodic with the measure of maximal entropy $\nu_{\eta'}\ast B(\frac12,\frac12)$. Hence this latter measure must be preserved by $\tau$. In particular, $\tau$ preserves its support. It is not hard to see that the support of $\nu_{\eta'}\ast B(\frac12,\frac12)$ is $X_{\eta'}$ and the result now follows from Theorem~\ref{t1}.
\end{proof}

In other words, in the general proximal case we reduced the problem of determining ${\rm Aut}(X_\eta,S)$ to the ``relative Behrend'' situation. This yields an intriguing question:
What is the automorphism group of a $\mathscr{B}$-free system in the Behrend case itself?\footnote{As shown in the recent article \cite{Dy-Ka-Ke}, in the minimal case, the group of automorphisms can be non-trivial.} It would be interesting to know whether the GoE theorem holds in this case.
It would also be interesting to know whether, in the hereditary case, all codes given by elements of ${\rm End}(X_\eta,S)$ are monotone. We were not able to decide whether, in the general case of a $\mathscr{B}$-free system, $\tau$ preserves the Mirsky measure, i.e.\ $\nu_\eta\circ \tau=\nu_\eta$,  for all $\tau\in {\rm Aut}(X_\eta,S)$.

\smallskip

It also seems natural to consider questions about a measure-theoretic GoE theorem. We give some samples of that. Assume that a CA $\tau$ preserves the Mirsky measure. Let $\tau$ be called $\nu_\eta$-pre-injective if it is injective on any collection of $0-1$-sequences that differ in a set of zero density. Is this notion equivalent to $\tau$ being onto $\nu_\eta$-almost surely? If yes, is there an analogous result for measures different from the Mirsky measure? May GoE proofs be more transparent when using isomorphism to the maximal equicontinuous generic factor  \cite{Ke-2020} or to the Kronecker factor?

\vspace{2ex}

\noindent
{\bf Acknowledgments}
CR would like to thank ML for an invitation to Toruń in September 2018, where this project had been initially discussed. Research of the first author is supported by Narodowe Centrum Nauki grant UMO-2019/33/B/ST1/00364.


\medskip

\noindent
Faculty of Mathematics and Computer Science,\\
Nicolaus Copernicus University, Chopin street 12/18, 87-100 Toruń, Poland,\\
mlem@mat.umk.pl, dsell@mat.umk.pl\\
Department of Mathematics,\\ University of Erlangen-N\"urnberg,
Cauerstraße 11, 91058 Erlangen, Germany,\\
keller@math.fau.de, christoph.richard@fau.de
\end{document}